\documentclass[11pt]{amsart}
\usepackage[dvips]{graphicx}
\usepackage{amsmath,graphics}
\usepackage{tikz}
\usepackage{amsfonts,amssymb}
\usepackage{xypic}
\usepackage{comment}
\usepackage{color}
\usepackage{hyperref}
\usepackage{cleveref}
\usepackage{thmtools}
\usepackage{thm-restate}

\specialcomment{proofc}{}{}
\includecomment{proofc}
\theoremstyle{plain}
\newtheorem*{theorem*}{Theorem}
\newtheorem*{lemma*} {Lemma}
\newtheorem*{corollary*} {Corollary}
\newtheorem*{proposition*}{Proposition}
\newtheorem*{conjecture*}{Conjecture}
\newtheorem{theorem}{Theorem}[section]
\newtheorem{lemma}[theorem]{Lemma}
\newtheorem*{theorem1*}{Theorem 1}
\newtheorem*{theorem2*}{Theorem 2}
\newtheorem*{theorem3*}{Theorem 3}

\newtheorem{proposition}[theorem]{Proposition}
\newtheorem{conjecture}[theorem]{Conjecture}
\newtheorem{question}[theorem]{Question}

\theoremstyle{remark}
\newtheorem*{remark}{Remark}
\newtheorem*{remarks}{Remarks}
\newtheorem*{definition}{Definition}

\newtheorem*{example*}{Example}

\theoremstyle{definition}

\def\op{\operatorname}

\def\G{\Gamma}

 \def\Q{\mathbb{Q}}  \def\Z{\mathbb{Z}}   
\def\N{\mathbb{N}}    
 \def\a{\alpha}   \def\bp{\begin{pmatrix}}
 \def\ep{\end{pmatrix}} \def\bn{\begin{enumerate}} 
   \def\en{\end{enumerate}}
\def\ba{\begin{array}} \def\ea{\end{array}}  
 \def\S{\Sigma}  \def\a{\alpha}  
 \def\Aut{\operatorname{Aut}} \def\Out{\operatorname{Out}}  
  
\def\ker{\mbox{Ker}}\def\be{\begin{equation}} \def\ee{\end{equation}}

\def\GG{\widetilde G}
\def\KK{\widetilde K}
\def\Ga{\widetilde \Gamma}
\def\XX{\widetilde X}

\def\op{\operatorname}

\def\modg{\op{Mod}_g}
\def\mod1{\op{Mod}_{g}^{1}}

\def\Mod{\op{Mod}}

\begin{document}
\title[Virtual algebraic fibrations of surface-by-surface groups]{Virtual algebraic fibrations of surface-by-surface groups and orbits of the mapping class group }

%Revision after referees' reports from Journal of Algebra
\author{Robert Kropholler, Stefano Vidussi and Genevieve Walsh}
%\address{Department of Mathematics, University of California,
%Riverside, CA 92521, USA} \email{svidussi@ucr.edu}
\date{\today}

\begin{abstract} We show that a conjecture of Putman--Wieland, which posits the nonexistence of finite orbits for higher Prym representations of the mapping class group, is equivalent to the existence of surface-by-surface and surface-by-free groups which do not virtually algebraically fiber. While the question about the existence of such groups remains open, we will show that there exist free-by-free and free-by-surface groups which do not algebraically fiber (hence fail to be virtually RFRS).

\end{abstract}

\maketitle

\section{Introduction}
In this paper we will be concerned with the study of homological and group-theoretical properties of group extensions of the form \begin{equation} \label{eq:ses}  1 \longrightarrow K \longrightarrow  G \stackrel{f}{\longrightarrow} \G \longrightarrow 1, \end{equation} where $K$ and $\G$, the {\em fiber} and {\em base} groups, are fundamental groups of closed, compact, orientable surfaces of genus at least $2$, or finitely generated nonabelian free groups. The study of this type of extension has quite a long story, originating in the realm of free-by-free groups and, in the broader context that includes surface groups, in several papers of F.E.A. Johnson, see e.g. \cite{Jo99} and references therein.

The original motivation of our interest stems from a topological framework. 
Let $X$ be a surface bundle with fiber $F$ over a surface $B$, both having genus at least $2$. Then its fundamental group $G = \pi_1(X)$ is an extension as in \Cref{eq:ses} with $K := \pi_1(F)$, $\Gamma := \pi_1(B)$ and where $f \colon G \to \G$  is the map induced on the fundamental groups (after picking a basepoint) by the fibration $X \rightarrow B$.  This sequence is the nontrivial part of the long exact sequence of the homotopy groups for the fibration. In this case, we will refer to $G$ as a {\it surface-by-surface group}. In the other cases, while the extensions are not the fundamental group of a closed aspherical $4$-manifold, they share similarities with surface-by-surface groups. Note that when the base $\G$ is the free group $F_n$, $G$ can always be thought of as a semidirect product $K \rtimes F_n$, while this may fail in general.

We will focus on those extensions with the property that the induced map on homology with rational coefficients $f_* \colon  H_1(G;\Q) \to H_1(\G;\Q)$ is an isomorphism.  (This map is always surjective.) This has been referred to by saying that the sequence of \Cref{eq:ses} has no {\em excessive homology} (\cite{KW19}). This is a property that depends on $G$ alone, and not on the choice of the extension. 

Given any group $G$ as in \Cref{eq:ses}, we can ask whether it admits a finite-index subgroup with excessive homology.   By the work of \cite{FV19,KW19} this is equivalent to the property that $G$ virtually algebraically fibers, namely a finite index subgroup of $G$ admits an epimorphism to $\Z$ with finitely generated kernel. Our main aim is to show that in the case where $K = \pi_1(F)$ is a surface group there is a natural relation between this question and a conjecture by Putman and Wieland  \cite[Conjecture 1.2]{PW13}. We refer the reader to Section \ref{sec:prelim} and the original source for more detail. The case labeled as NFO($g,0,0$) of that conjecture posits that the action  of the mapping class group $\Mod^{1}_{g}$ of the surface $F$ (where $g$ is the genus of $F$) on the rational homology of finite-index characteristic covers of $F$ (sometimes referred to as {\em higher Prym representation}) has no finite orbits. We will show the following.

\begin{restatable*}{theorem}{mainthm}\label{thm:main}
	For every $g \geq 2$  the Putman--Wieland conjecture NFO($g,0,0$) holds if and only if there exists a surface-by-surface or a surface-by-free group $G$ with fiber of genus $g$ and no virtual excessive homology or, equivalently, that is not virtually algebraically fibered. 
\end{restatable*} 

Note that the surface-by-surface and surface-by-free groups mentioned in the statement, if they exist, will fail to be virtually RFRS (residually finite rationally solvable) in light of \cite{K20}.

Assuming the Putman-Wieland conjecture, this theorem would differentiate the behavior of $3$- and $4$-dimensional surface bundles. 

Markovi\'c \cite{Mark22} has shown that the conjecture fails when the genus of the surface is $2$. Combining this with Theorem \ref{thm:main}, we have the following, that gives a partial answer to a question of Hillman (\cite[Section 11, Question 4]{Hi15}):

\begin{restatable*}{corollary}{corgen}\label{cor:g2}
Let $X$ be a surface bundle with fiber $F$ of genus $2$ and base $B$ of genus at least $2$. Then its fundamental group $G = \pi_1(X)$ is virtually algebraically fibered and it is incoherent, namely it admits finitely generated subgroups which are not finitely presented.
\end{restatable*}

The proof of Theorem \ref{thm:main} shows that the base of $G$ can be assumed to be the fundamental group of a surface of genus $2$, or the free group $F_2$. The fact that these bases are optimal is not quite obvious, see Proposition \ref{pr:torus}. Additionally, the corresponding surface bundle over a surface $X$ can be assumed to have signature zero.

The ``only if" part of Theorem \ref{thm:main} uses an epimorphism from the base of the extension to the mapping class group $\Mod_g$, so that we can make the construction of the groups deciding the conjecture NFO($g,0,0$) very explicit, especially in the case where $g \geq 3$. We illustrate this in the surface-by-free case: Let $F$ be a surface of genus $g \geq 3$; this admits a cyclic automorphism of order $4g+2$, and denote by $\Pi$ the fundamental group of the mapping torus of this automorphism. (This is a Seifert--fibered manifold which is finitely covered by a product.) Any two such automorphisms are conjugate in $\Mod_g$, so that $\Pi$ is uniquely determined as a group. By \cite{Kor05} $\Mod_g$ can be generated by two such automorphisms, related by conjugation by the automorphism $\delta \colon K \to K$ induced by a Dehn twist along a nonseparating curve. (See Section \ref{sec:mainresults} for references and more details), 

\begin{restatable*}{proposition}{simplenofiber}\label{prop:simplenofiber}
	Let $\upsilon \colon K \to K$ be the generator of a cyclic subgroup of order $4g+2$ of $\modg$ for $g \geq 3$, and denote by $\Pi$ the corresponding mapping torus. Let $\Pi *_{\delta} \Pi$ be the amalgamated free product determined by the automorphism $\delta \colon K \to K$. Then the Putman--Wieland conjecture NFO($g,0,0$) holds if and only if the surface-by-$F_2$ group $\Pi *_{\delta} \Pi$ fails to virtually algebraically fiber.
\end{restatable*}

The group $\Pi \ast_{\delta} \Pi$ surjects onto $\mod1$. In a sense, the latter is ``one Dehn twist away" from being the product $K \times \modg$. 

In analogy with Theorem \ref{thm:main} one may ask about the existence of free-by-free or free-by-surface groups with no virtual excessive homology, or equivalently which do not virtually algebraically fiber. In this realm we can reach in most cases an affirmative answer, which naturally extends also to the case where the fiber group is free abelian.

\begin{restatable*}{theorem}{freebyfreenv} \label{thm:freebyfreenv} For each $n \geq 2$,  and each $m \geq 4$ (respectively $m \geq 2$) there exist groups of the form $ F_{m} \rtimes \G$ (respectively $\Z^{m} \rtimes \G$), where $\G$ is a copy of $F_n$ or the fundamental group of a surface of genus $n$, with no virtual excessive homology or, equivalently, that are not virtually algebraically fibered. Therefore these groups are not virtually RFRS.
\end{restatable*}

Note, in contrast, that these groups are virtually residually $p$, hence virtually residually finite solvable, see Lemma \ref{lm:RFS}. The interest in RFRS groups stems from their role in various areas of group theory, including the study of $3$--manifold groups (\cite{Ag08}) and algebraic fibrations (\cite{K20}). Free groups are RFRS, and so are their direct products: it is not obvious how to decide if other extensions of free (or surface) groups are RFRS. Theorem \ref{thm:freebyfreenv} as well as the previous results, makes a step in that direction.

We point out that Theorem \ref{thm:freebyfreenv} has no implication on the suitable analog of the Putman-Wieland conjecture in the context of automorphisms of free groups, which is known to be true in light of \cite{HF16}.

This paper is organized as follows: In Section \ref{sec:prelim}, we give background on the homology of the extensions we are considering, and of their finite index subgroups, and state the Putman-Wieland conjecture. \Cref{thm:main} and  \Cref{prop:simplenofiber} are proven in Section \ref{sec:mainresults}, while Theorem \ref{thm:freebyfreenv} is proven in Section \ref{sec:fbf}.  

\subsection*{Acknowledgements:}
The first author was funded by the Deutsche Forschungsgemeinschaft (DFG, German Research Foundation) under Germany's Excellence Strategy EXC 2044--390685587, Mathematics M\"unster: Dynamics--Geometry--Structure. The second author was partially supported by the Simons Foundation Collaboration Grant For Mathematicians 524230. The third author was partially supported by National Science Foundation grant DMS - 2005353. The authors wish to thank the referees for their excellent reports, which greatly improved the presentation.

\section{preliminaries} \label{sec:prelim} 

It will be useful in what follows to tie the sequence in \Cref{eq:ses} with the monodromy representation determined by the extension (\ref{eq:ses}). Namely, we have
\begin{equation} \label{eq:bir} 
\xymatrix{
K  \ar@{>->}[r]\ar[d]^{\cong} &
G \ar[d]^{\zeta} \ar@{->>}[r]^{f}& \G \ar[d]^{\eta}\\
K \ar@{^{(}->}[r]& \Aut(K) \ar@{->>}[r]^{p}  & \Out(K). }
\end{equation}
Here $\eta$ characterizes the action of $\Gamma$ on $K$, well-defined up to conjugation. The map $K \to \Aut(K)$ is given by the conjugation action, and it is injective as $K$ has trivial center.  This allows us to identify $G$ as the pullback 
\[ G \cong \{  (\psi,\gamma) \in \Aut(K) \times \G \mid p(\psi) = \eta(\gamma) \}, \] with the group structure obtained by restriction of that on $\Aut(K) \times \G$.  The fibration map $f$ is induced by projection onto the second factor, and the fiber subgroup is given by the normal subgroup \[  K \times \{1_{\G}\} \unlhd G \leq  \Aut(K) \times \G. \]

With this identification, the conjugation action of  $G$ on its normal subgroup $K$ can be written in terms of the conjugation action of $\Aut(K) \times \G$ on  $K \times \{1_{\G}\}$. Note that when $K$ is a surface group of genus $g>1$, we will be interested in the case where the monodromy representation $\eta \colon \G \to \Out(K)$ has values in $\modg \unlhd \Out(K)$, and $\zeta \colon G \to \Aut(K)$ has values in $\modg^1 \unlhd \Aut(K)$ (with both modular groups subgroups of index $2$ determined by orientation-preserving homeomorphisms of a surface). This condition, which is equivalent to the fact that the corresponding surface bundle is oriented, is not too restrictive and can be achieved by passing to an index-two subgroup of the base of the extension. In such case, the bottom row of Equation (\ref{eq:bir}) can be interpreted as Birman's short exact sequence.

\subsection{Coinvariants and homology} We wish to understand the homology of $G$ in terms of that of $K$ and $\G$. If $G$ is a group and $M$ is a $G$-module, the {\it coinvariants of $M$}, denoted $M_G$, is the quotient of $M$ obtained by taking the quotient generated by elements of the form $gm-m$, for all $g \in G$, $m \in M$. 
The {\it invariants of $M$}, denoted $M^G$, are the largest submodule of $M$ on which $G$ acts trivially. 

\begin{definition} Let $1 \to K \to G \to \Gamma \to 1$ be an extension. The {\it excessive homology} of this extension is the kernel of the map $H_1(G; \Q) \rightarrow H_1(\Gamma;\Q)$.
\end{definition} 

In the cases that we will be interested in, the excessive homology can be conveniently expressed in terms of (co)invariant homology. Specifically, we have the following Lemma (whose first part applies, in generality, for any semidirect product.)

\begin{lemma} \label{lem:ehsd} Let $G$ be an extension as in Equation (\ref{eq:ses}); assume that either
\begin{enumerate}
    \item $G$ can be written as a semidirect product $G = K \rtimes \G$; or
    \item $K$ is a surface group (of genus at least $2$).
\end{enumerate} 
Then the excessive homology of $G$ is $H_1(K;\Q)_G$. Furthermore, in the second case, it is isomorphic to $H_1(K; \Q)^G$.\end{lemma}

\begin{proof} The Lyndon--Hochschild--Serre spectral sequence associated to  \Cref{eq:ses} \cite[problem 6, pg. 47]{Br94}, gives the following 5-term exact sequence for the homology with rational coefficients: \begin{equation} \label{eq:five} H_2(G;\Q) \stackrel{f}{\longrightarrow} H_2(\G;\Q) \longrightarrow H_1(K;\Q)_{G} \longrightarrow H_1(G;\Q) \stackrel{f}{\longrightarrow} H_1(\G;\Q) \longrightarrow 0.  \end{equation} When $G$ is a semidirect product, the extension splits, hence all maps $f \colon H_i(G;\Q) \to H_{i}(\G;\Q)$ admit a right-inverse, hence are surjective, and the statement follows. When $K$ is a surface group, and $\G$ is free, then the extension is split again and the previous argument applies. Surface-by-surface groups, in contrast, do not always split.  However, as we assumed that the genus of the fiber is at least $2$, the map $f \colon H_2(G;\Q) \longrightarrow H_2(\G;\Q)$ in the sequence in Equation (\ref{eq:five}) is surjective (see e.g. \cite[Proposition 3.1]{Mo87}), so the kernel of $f \colon H_1(G;\Q) \to H_1(\G;\Q)$ is given again by $H_1(K;\Q)_G$.

Finally, when $K$ is a surface group, $G$ acts on $H_{1}(K;\Z)$ preserving its algebraic intersection form, which extends uniquely to a $G$-invariant symplectic structure on $H_1(K;\Q)$ (see e.g. \cite[Section 6.1.2]{FM12}). The symplectic structure on $H_{1}(K;\Q)$, being non-degenerate, induces an isomorphism of $G$-spaces between $H_{1}(K;\Q)$ and its dual $Hom(H_1(K;\Q),\Q) \cong H^{1}(K;\Q)$. Consequently, $H^{1}(K;\Q)^{G} \cong H_{1}(K;\Q)^{G}$.
But $H^{1}(K;\Q)^{G}$ is dual to $H_{1}(K;\Q)_{G}$ (see e.g. \cite[Lemma 2.1]{PW13}), hence the latter is isomorphic (as vector spaces) to $H^{1}(K;\Q)^{G}$ so the last part of the statement follows.
\end{proof}

It is not too hard to verify that the existence of excessive homology is preserved by passing to finite index subgroups:

\begin{lemma} \label{lem:fi} Let $1 \to K \to G \to \Gamma \to 1$ be an extension as in Equation (\ref{eq:ses}), and let $\GG \leq  G$ be a finite index subgroup. Then the excessive homology of the induced extension on $\GG$ surjects onto that of $G$.  \end{lemma}

\begin{proof} There exists a commutative diagram 
\begin{equation} \label{eq:finco} \xymatrix{
\KK = K \cap \GG \ar@{^{(}->}[r]\ar@{^{(}->}[d] &
\GG \ar@{^{(}->}[d] \ar@{->>}[r]& \Ga \ar@{^{(}->}[d]\\
K \ar@{^{(}->}[r]& G \ar@{->>}[r]  &
\G } \end{equation} 
with self-explaining notation. This entails the existence of a homomorphism 
\[ H_{1}(\KK;\Q) \longrightarrow H_{1}(K;\Q) \] which naturally commutes with the action of $\GG$. As $\KK \leq K$ is finite index, this homomorphism admits a right-inverse, a suitable multiple of the transfer map, see e.g. \cite[Chapter III]{Br94}). (This necessitates the use of rational coefficients.), and it is therefore surjective. It follows that we  have a composition of epimorphisms 
\[ H_{1}(\KK;\Q)_{\GG} \longrightarrow H_{1}(K;\Q)_{\GG} \longrightarrow H_{1}(K;\Q)_{G} \]  
which means that the excessive homology of the sequence $1 \to K \to G \to \G \to 1$ cannot decrease by passing to finite index subgroups of $G$. 
\end{proof}

Lemma \ref{lem:fi} implies that we can reduce the question of the existence of virtual excessive homology to normal finite index subgroups. Therefore we will focus on these subgroups, which are determined by an epimorphism $\a\colon G \to S$ onto a finite group $S$, completing the sequence in \Cref{eq:ses} to the following commutative diagram of short exact sequences: 
\begin{equation} \label{eq:diag} 
\xymatrix{
\KK \ar@{^{(}->}[r]\ar@{^{(}->}[d] &
\GG \ar@{^{(}->}[d] \ar@{->>}[r]& \Ga \ar@{^{(}->}[d]\\
K \ar@{->>}[d] \ar@{^{(}->}[r]& G \ar@{->>}[d] \ar@{->>}[r]  &
\G \ar@{->>}[d] \\
\a(K) \ar[r]
 \ar[r]& S \ar[r] &
S/ \a(K).} 
\end{equation} 
\begin{remark} In general, the action of $\GG$ on $H_1(\KK;\Q)$ does not extend to $G$, but it will do so when $\KK \unlhd K$ is characteristic: such action is induced by restriction of the $G$-action on $K$, which preserves $\KK$ as well as $[\KK,\KK]$ whenever $\KK$ is characteristic. Conversely, given any finite index normal subgroup $\KK \unlhd K$, it is known from \cite[Lemma 4.1]{Mo87} that there exists a finite index normal subgroup $\GG \unlhd G$ whose intersection with $K$ is the subgroup $\KK \unlhd K$. (\cite[Lemma 4.1]{Mo87} is stated only for surface groups, but the proof applies to free groups as well.)   Again, the action of $\GG$ on $H_1(\KK;\Q)$ may fail to extend to $G$, but will do so whenever $\KK$ is characteristic. As any finite index normal subgroup $\KK \unlhd K$ contains a finite index subgroup which is characteristic in $K$, we can further reduce the study of excessive homology to covers of $X$ which induce fiberwise characteristic subgroups of $K$.
\end{remark} 
\section{Virtual excessive homology and the Putman--Wieland conjecture} \label{sec:mainresults}

\subsection{The Putman-Wieland Conjecture}

In \cite{PW13}, the authors connect the study of the orbits of the mapping class group acting on the first homology of a surface and its finite covers to the classical conjecture that mapping class groups do not virtually surject $\Z$. We will be interested only in a particular case of their conjecture, so for sake of simplicity we will limit ourselves to that case, referring the reader to \cite{PW13} for the general case. Denote by $\mod1$ (respectively $\modg$) the mapping class group of a surface $\S_{g}^1$ with genus $g$ and $1$ puncture (respectively of a surface $\S_g$ with no punctures). When $\KK \unlhd K =  \pi_1(\S_{g})$ is a characteristic subgroup, the group $\mod1$ acts on $H_1(\KK;\Q)$. Putman and Wieland posit the following: 

\begin{conjecture} \cite[Conjecture 1.2]{PW13} Fix $g \geq 2$. Let $\KK \unlhd K = \pi_1(\Sigma_{g})$ be a finite-index characteristic
subgroup. Then for all nonzero vectors $v \in H_1(\KK;\Q)$ the $\mod1$
 -orbit of $v$ is infinite. \end{conjecture}
 
For fixed $g \geq 2$, the conjecture above will be referred to, for consistency with the notation of \cite{PW13}, as NFO($g,0,0$)

\subsection{Connections with virtually excessive homology}

Our first result connects virtual excessive homology of a surface-by-surface or surface-by-free group $G$ with the the orbits on the homology of characteristic subgroups of the fiber group $K$. 

\begin{lemma} \label{lem:equi} Let $G$ be a surface-by-surface or a surface-by-free group. Then the following two properties are equivalent:
\bn
\item $G$ has no virtually excessive homology;
\item for any finite index characteristic subgroup $\KK \unlhd K$, and any nonzero $v \in H_{1}(\KK;\Q)$, the orbit $G \cdot v \subset  H_{1}(\KK;\Q)$ is infinite.
\en \end{lemma} 

\begin{proof} (1) $\Rightarrow$ (2): Assume that there exists a finite index characteristic subgroup $\KK \unlhd K$, and a nonzero $v \in H_{1}(\KK;\Q)$ with {\em{finite}} orbit $G \cdot v$. As discussed before, $G$ admits a normal finite index subgroup $\GG \unlhd_{f} G$ whose intersection with $K$ is the subgroup $\KK \unlhd K$.
As $G \cdot v$ is finite, there is a finite index subgroup $H \leq  G$ such that the orbit $H \cdot v = \{v\} $. As $\KK$ acts trivially on its own homology, $\KK \leq H$. 
We now replace $\GG$ with its finite index subgroup $H \cap \GG$; namely, we take a finite index subgroup of $\GG$ defined by the pull-back of a suitable finite index subgroup of $\widetilde \G$.
Going to the normal core of $H \cap \GG$ in $G$ if needed, we can assume $H \cap \GG \leq  G$ is normal; this normal core will still contain $\KK$ by the assumption that the latter is characteristic in $K$.  

Hoping that no risk of confusion arises, we maintain the notation $1 \to \KK \to \GG \to {\widetilde \G} \to 1$  for the ensuing finite index normal subgroup of $G$: we stress that  $\KK$ has not changed in the process. With this notation in place, we have that $\GG \cdot v = \{v\}$, whence the space of invariants $H_{1}(\KK;\Q)^{\GG}$ is nontrivial. As $\GG$ acts on $H_{1}(\KK;\Q)$ preserving its intersection form, the symplectic structure on $H_{1}(\KK;\Q)$ induces an isomorphism of $\GG$-spaces between $H_{1}(\KK;\Q)$ and its dual $H^{1}(\KK;\Q)$; consequently, $H^{1}(\KK;\Q)^{\GG} \cong H_{1}(\KK;\Q)^{\GG}$. The latter vector space is nontrivial, as it contains the span of  $v$.
But $H^{1}(\KK;\Q)^{\GG}$ is dual to $H_{1}(\KK;\Q)_{\GG}$ (see e.g. \cite[Lemma 2.1]{PW13}), hence
$\mbox{dim\hspace*{1pt}} H_{1}(\KK;\Q)_{\GG} > 0$.  It follows that $\GG$ has excessive homology. 

(2) $\Rightarrow$ (1): Let $\GG$ be any finite index subgroup of $G$: we want to show that if (2) holds, then $\GG$ has no excessive homology. First, as the existence of excessive homology is preserved by passing to finite index subgroups, we can assume that $\GG$ is normal in $G$. Denote now $\KK = K \cap \GG$. Any surface group $\KK$ contains a finite index subgroup characteristic in $K$. Replacing $\GG$ if necessary with its normal finite index subgroup whose intersection with $\KK$ is that characteristic subgroup (see the Remark in Section \ref{sec:prelim}), we can assume without loss of generality that $\KK$ is characteristic itself in $K$.  
By assumption, for any nonzero $v \in H_{1}(\KK;\Q)$ the orbit $G\cdot v$ is infinite. As $\GG \unlhd G$ is finite index, so must be the orbit $\GG \cdot v$. The space of invariants $H_1(\KK;\Q)^{\GG}$ is trivial, and proceeding as above so is the space of coinvariants $H_1(\KK;\Q)_{\GG}$, hence $\GG$ has no excessive homology. \end{proof}

Next, we show the equivalence of the existence of a surface-by-surface or surface-by-free group with no virtual excessive homology and the case NFO($g,0,0$) (Conjecture 1.2 of \cite{PW13} for $\Sigma_{g,0}^{0}$) of the Putman--Wieland Conjecture.

\mainthm

\begin{proof} Let $G$ be an extension as in the statement.  We claim that for any finite index characteristic subgroup $\KK \leq  K$, and any nonzero $v \in H_{1}(\KK;\Q)$ the orbit $\mod1 \cdot v \subset  H_{1}(\KK;\Q)$ is infinite. Indeed, by Lemma \ref{lem:equi}, we know that for any finite index characteristic subgroup $\KK \leq  K$ and any nonzero $v \in H_{1}(\KK;\Q)$ the orbit $G \cdot v \subset  H_{1}(\KK;\Q)$ is infinite, which shows that the orbit $\mod1 \cdot v \supseteq G \cdot v$ is infinite as well. 

To prove the reverse implication, we recall that, given any finite presentation $\Theta$ of $\Out(K)$, there exists a surface bundle $X$ of fiber $F$ over a surface $B$ (whose genus equals the rank $r$ of the presentation) induced by an epimorphism $\eta \colon \G \to  \modg$. This construction is due to Kotschick in \cite[Proposition 4]{Ko98}; the map $\eta$ is defined by sending the first $r$ generators of $\G = \pi_1(B) = \langle \alpha_1,...,\alpha_{r}, \beta_1,...,\beta_{r}|\prod_{i=1}^{r} [\alpha_i,\beta_i]\rangle$ to the set of generators of $\Out(K)$, while the remaining $r$ generators are sent to the trivial element. We denote this surface bundle by $X_\Theta$ so that $G = \pi_1(X_\Theta)$ will be the desired surface-by-surface group. Similarly, we can consider the presentation epimorphism $\eta \colon \G = F_r \to \modg$, with $G$ being the induced surface-by-free group. By construction, these extensions are of type I in Johnson's trichotomy (namely, the monodromy homomorphisms $\eta \colon \G \to \modg$ have infinite kernel and image, see \cite{Jo93}) and, in the surface-by-surface case $X_\Theta$ has signature zero.  The virtual excessive homology is determined by the behavior of the orbits of $G$ on the homology of the characteristic subgroups of $K$. As $\eta \colon \G \to \modg$ is surjective, so is $\zeta \colon G \to \mod1$, so the $G$-orbits coincide with the orbits of $\mod1$. It follows that NFO($g,0,0$) is true if and only if $G$ has no virtual excessive homology.  Note that as there exists presentations of $\modg$ with two generators (\cite{Wa96,Kor05}), we can assume that $r = 2$, in particular the base $B$ of the surface bundle over a surface can be chosen to have genus $2$.
\end{proof}

It is quite straightforward to see that the result above is optimal as far as the base genus is concerned. In fact we have the following: 

\begin{proposition} \label{pr:torus} Let $F \hookrightarrow X \stackrel{f}{\to} T^2$ be a surface bundle over a torus with fiber genus greater or equal than $2$. Then $G$ has nonzero virtual excessive homology. \end{proposition}
\begin{proof} 
	In \cite{FV13} it is proven that the fundamental group of a surface bundle of the type described is large. In particular, this implies that $vb_1(X) = \infty$. Then let $\XX \to X$ be a finite regular cover, as described in the diagram of  \Cref{eq:diag}, with $b_1(\XX) > 2$.  As ${\widetilde B}$ is also a torus, the fibration ${\widetilde F} \hookrightarrow \XX \to {\widetilde  B}$ has excessive homology.
\end{proof} 

In \cite{Mark22} the author proves that NFO($2,0,0$) fails. (As often happens, the case of genus $2$ is special because the mapping class group is entirely composed of hyperelliptic classes.) This entails the following:

\corgen

\begin{proof} The existence of a virtual algebraic fibration for $G$ follows immediately from Theorem \ref{thm:main}. The proof of incoherence is standard, and we reproduce it for completeness. As coherence is preserved by subgroups, we can assume that $G$ itself is algebraically fibered. Let $ \phi \colon G \to \Z$ be an algebraic fibration and let \[ 1 \longrightarrow \ker \,\phi \longrightarrow G \longrightarrow \Z \longrightarrow 1 \] be the corresponding short exact sequence, where $\ker \,\phi \unlhd G$ is finitely generated. By \cite[Theorem 4.5(4)]{Hi02} the group $\ker \,\phi$ can have finiteness of type $FP_2$ (in particular, be finitely presented) only if the Euler characteristic of $G$ is zero. As both base and fiber groups of $G$ have nonzero Euler characteristic, this fails, hence $\ker \,\phi$ is finitely generated but not finitely presented.
\end{proof}

It is quite interesting, at this point, to ask whether there exist classes of surface bundles for which there is always virtual excessive homology, and hence virtual algebraic fibrations. For instance, this is the case when the fibration is a holomorphic bundle, see e.g. \cite{BHPV04}. In particular, it would be interesting to decide this case for the class of Kodaira fibrations, or of surface bundles of type III in Johnson's trichotomy (injective monodromy). (The surface bundles discussed in Theorem \ref{thm:main} cannot be Kodaira fibrations, as these have strictly positive signature, see e.g. \cite{BHPV04}.)

In the case where $g \geq 3$, we can use a result of Korkmaz to give a quite explicit description of the type of surface bundles that are involved in the statement of \Cref{thm:main}. For sake of concreteness, we limit ourselves to the surface-by-free case, that is somewhat more striking. In \cite[Section 5]{Kor05} Korkmaz shows that the mapping class group $\modg$ can be assumed to be generated by two elements, generators of two cyclic subgroups of order $4g + 2$ of $\modg$. These generators are conjugated in $\modg$ by a Dehn twist along a nonseparating curve. We'll denote by $\delta \colon K \to K$ the corresponding automorphism.
We have the following:

\simplenofiber
\begin{proof} The proof of this proposition is a specialization of an argument used in the proof of Theorem \ref{thm:main}. By \cite{Kor05} there exists a presentation of  $\modg$ of rank $2$ in which the two generators $x,y$ each generate a cyclic subgroup of order $4g+2$. All these generators are conjugate in $\modg$ (see e.g. \cite[Section 7.2.4]{FM12}), so that in particular the mapping tori of the induced automorphisms of $K$ are isomorphic. In the case at hand we can assume that the conjugating element is induced by a Dehn twist along a nonseparating curve (see \cite[Section 5]{Kor05}). It follows that there exists a commutative diagram of the form 
\[  \xymatrix{
K \ar@{^{(}->}[r]\ar[d]^{\cong} &
G \ar@{->>}[d]^{\zeta} \ar@{->>}[r]^{f}& F(x,y) \ar@{->>}[d]^{\eta} \\
K \ar@{^{(}->}[r]& \mod1 \ar@{->>}[r]^{p}  &
 \modg } \]
 where $\eta \colon F(x,y) \to \modg$ is the presentation quotient. By \textit{fiat}, $G$ is the free product of the mapping tori of two automorphisms of $K$, amalgamated along $K$. These mapping tori arise as the pull-back of the monodromies determined by the two generators $x$ and $y$ of  $F(x,y)$, namely they are the unique (up to conjugation) cyclic monodromies of order $4g+2$ on $K$. Denoting by $\Pi$ the resulting mapping torus, well-defined up to isomorphism, the group $G$ is isomorphic to the free amalgamated product $\Pi *_{\delta} \Pi$, where the amalgamation is determined by the automorphism $\delta \colon K \to K$. The rest of the proof follows exactly as the proof of \Cref{thm:main}.  \end{proof}

As a remark, note that the five term sequence of the Lyndon--Hochschild--Serre spectral sequence tells us that $H_1(\Pi) = \Z \oplus H_{1}(K)_{\Z_{4g+2}}$ where we make explicit that the periodic monodromy factorizes through the quotient map $\Z \to \Z_{4g+2}$. The action of $\Z_{4g+2}$ on $K$ determines an orbisurface cover (or, if preferred, a branched cover) whose quotient is an orbisphere with $3$ orbifold points of order $2,2g+1, 4g+2$ whose orbifold fundamental group we denote $\Delta$. We have a short exact sequence
\[  1 \longrightarrow K \longrightarrow  \Delta \longrightarrow \Z_{4g+2} \longrightarrow 1.\]
As $\Z_{4g+2}$ is torsion, the coinvariant homology $H_{1}(K)_{\Z_{4g+2}}$ has the same rank as $H_1(\Delta)$, namely it is torsion. It follows that $b_1(\Pi) = 1$ and $b_1(\Pi *_{\delta} \Pi) = 2$.

\section{Extensions with no virtual excessive homology}\label{sec:fbf} 

In this section we will show that there are extensions of free (and free abelian) groups that have no virtual excessive homology. As we already observed, this does not imply an analog of the NFO conjecture in the realm of free groups. However, it is interesting that the proof hinges on Property (T) for suitable automorphism groups, a theme related with the circle of ideas at the origin of \cite{PW13}. Regarding Property (T), we recall the definitions here; for full details see \cite{bdlhv}.   

\begin{definition}\cite[Def 1.1.3, 1.4.3]{bdlhv}
	A group $G$ has {\em  Property (T)}, if every unitary representation of $G$ with almost invariant vectors has a non-trivial invariant vector.  A pair of discrete groups $(G, H)$ with $H\leq G$ has {\em Relative Property (T)}, if every unitary representation of $G$ with almost invariant vectors has a non-trivial $H$ invariant vector. 
\end{definition}

For the precise definition of almost invariant vectors see Definition 1.1 in \cite{bdlhv}.  If $G$ has Property (T), then any quotient of $G$ also has Property (T). As discrete, finitely generated amenable groups with Property (T) are finite, any amenable discrete quotient of $G$ must be finite.  Likewise, if $(G, H)$ has Relative Property (T), then in any amenable discrete quotient of $G$ the image of $H$ must be finite.
This follows since an amenable group $K$ has almost invariant vectors in the left-regular representation on $\mathcal{L}^2(K)$ (Reitner's condition). 
In particular, this holds for abelian quotients, since abelian groups are amenable. 

\freebyfreenv

\begin{proof}  
	In \cite{KNO19, KKN21, Nit20}, it is shown that $\Aut(F_m)$ has property (T) for $m\geq 4$. 
	Thus by \cite[Prop 10]{KNO19} we see that $F_m\rtimes \Aut(F_m)$ has property (T). 
	Thus if $Q$ is an abelian quotient of a finite index subgroup $H$ of $G = F_m\rtimes\Aut(F_m)$, then $Q$ is finite.  In particular, any finite index subgroup of $G$ has finite abelianization. 
	Similarly, the pair $(\Z^m\rtimes SL_m(\Z), \Z^m)$ has Relative Property (T) when $m \geq2$ \cite[Ex. 1.7.4, 4.2.2]{DLHV89}, thus given any abelian quotient $Q$ of a finite index subgroup $H$ of $\Z^m\rtimes SL_m(\Z)$, then $\Z^m\cap H$ has finite image in $Q$.   These are the key properties that make our proof work. 
	
	We start with the case where the fiber $K$ is the free group $F_m$ for $m \geq 4$. Let $\eta \colon \Gamma \to \Aut(F_m)$ be a surjection and build the associated extension $G = F_m \rtimes \G$. Note we can take $\Gamma = F_n, \pi_1(S_g)$ with $g=n=2$ since $Aut(F_m)$ is generated by 2 elements \cite{N33}. Thus we have a commutative diagram 
	\[ \xymatrix{
F_m \ar@{^{(}->}[r]\ar[d]^{\cong} & G \ar@{->>}[d]^{\zeta} \ar@{->>}[r]^{f}& \G \ar@{->>}[d]^{\eta} \\
F_m \ar@{^{(}->}[r]&  F_m\rtimes \Aut(F_m) \ar@{->>}[r]^{\hspace*{10pt} p}  &
 \Aut(F_m) } \]

 The top row of this diagram is split so we have a splitting $s\colon \Gamma \to G$. 
 
	The proof proceeds with a variation on the proof of Theorem \ref{thm:main}, as we need to control the virtual coinvariant homology of $G$ without resorting to the invariant homology. 
	
	Let ${\widetilde G}$ be an arbitrary finite index subgroup of $G$. Then we have the short exact sequence 
	$$ 1 \rightarrow F_k \rightarrow {\widetilde G} \rightarrow {\widetilde \G} \rightarrow 1 $$ 
	where ${\widetilde \G}$ is the image of ${\widetilde G}$ under $f \colon G \rightarrow \G$ and $F_k = {\widetilde G} \cap F_m$. We note that ${\widetilde \Gamma}$ may not stabilize $F_k$ under the action given by $\eta({\widetilde \Gamma})$.  If this is the case, then  $s({\widetilde \Gamma})$ is not contained in  ${\widetilde G}$.  However, $s({\widetilde \G}) \cap {\widetilde G}$ has finite index in $s({\widetilde \G})$.  Therefore we can consider the subgroup ${\widehat G}$ of $G$ generated by $F_k$ and $s(\G)\cap {\widetilde G}$;
	this has finite index in ${\widetilde G}$, hence in $G$ as well. Thus we obtain a commutative diagram as follows:  
	\begin{equation} \label{eq:fbb} \xymatrix{
F_k \ar@{^{(}->}[r]\ar[d]^{\cong} & {\widehat G} \ar@{->>}[d]^{\zeta} \ar@{->>}[r]^{f}& {\widehat \G} \ar@{->>}[d]^{\eta} \\
F_k \ar@{^{(}->}[r]&  H \ar@{->>}[r]^{p}  &
 T } \end{equation} 
 where $H = \zeta({\widehat G})$ is finite index in $F_m \rtimes \Aut(F_m)$; furthermore
 ${\widehat \G} = f(\widehat{G}) = f(s(\G)\cap {\widetilde G})$ is finite index in $\G$ and $T$, the image of $\widehat{\G}$ under $\eta$, is finite index in $\Aut(F_m)$.

By construction, the top horizontal sequence in Eq (\ref{eq:fbb}) splits: any element $\gamma \in f(s(\G) \cap \GG)$ is mapped to $s(\gamma) \in \GG$, which sits by construction in $\widehat{G}$. The image $H$ is also a semidirect product; it is the image of $F_k \rtimes \widehat{\G}$ in $F_m \rtimes Aut(F_m)$ where the action of $\zeta(s(\widehat{\G}))$ on $F_m$ stabilizes $F_k \leq F_m$. 

So, as excessive homology is non-decreasing over finite index subgroups by Lemma \ref{lem:fi}, we can restrict ourselves to the case where the finite index subgroup of $G$ is a semidirect product of the form $\widehat{G} = F_k \rtimes \widehat{\G}$ for some finite index subgroup $\widehat{\G} \leq \G$, and its image $H \leq F_{m} \rtimes \Aut(F_m)$ is an extension of $F_k$ by a finite index subgroup $T \leq \Aut(F_m)$. By Lemma \ref{lem:ehsd} the excessive homology of ${\widehat G}$ is given by the coinvariant homology group $H_1(F_k;\Q)_{\widehat{\G}}$, and as the action of ${\widehat \G}$ on $F_k$ factors through $T$ by construction, this coinvariant homology group coincides with $H_1(F_k;\Q)_T$. By Lemma \ref{lem:ehsd} again, the excessive homology of ${H}$ is given by $H_1(F_k;\Q)_{T}$ as well. 

At this point we can invoke the fact that abelian quotients of any finite index subgroups of $F_m \rtimes \Aut(F_m)$, in particular ${H}$, are finite. This implies that $H_1(H;\Z)$ is torsion, hence $H_1(F_k;\Q)_{T}$ is trivial (and $b_1(\widehat{G}) = b_1(\widehat{\G})$).

	The proof for $K = \Z^m$ is slightly different from the above, although we could have done the proof above using Relative Property T, see remarks below. The pair $(\Z^m\rtimes SL_m(\Z), \Z^m)$ has Relative Property (T) when $m \geq2$ \cite[Ex. 1.7.4, 4.2.2]{DLHV89}, thus given any abelian quotient $Q$ of a finite index subgroup $H$ of $\Z^m\rtimes SL_m(\Z)$, then $\Z^m\cap H$ has finite image in $Q$. 
	
	It is classically known that there exist presentations of $SL_m(\Z)$ with two generators (see \cite{N33}); we can therefore again choose an epimorphism $\eta \colon \G \to SL_m(\Z)$ where $\Gamma$ is either a free or a surface group of rank or genus at least $2$ and consider the extension $G$ defined as the pull-back of the semidirect product $\Z^m\rtimes SL_m(\Z)$ under the projection onto the base. The groups in question fit in the commutative diagram.

\begin{equation} \xymatrix{
\Z^m \ar@{^{(}->}[r]\ar[d]^{\cong} &
G \ar@{->>}[d]^{\zeta} \ar@{->>}[r]^{f}& \G \ar@{->>}[d]^{\eta} \\
\Z^m \ar@{^{(}->}[r]&  \Z^m\rtimes SL_m(\Z) \ar@{->>}[r]^{p}  &
 SL_m(\Z) } \end{equation} 
Given a finite index subgroup $\widehat{G}$ of $G$, we can assume as above that $\widehat{G}$ has the  form $\widehat{K} \rtimes \widehat{\G}$ for some finite index subgroup
 $\widehat{\G} \leq \G$ and finite index $\widehat{K} \leq \Z^m$ (where of course $\widehat{K}$ itself is abstractly isomorphic to $\Z^n$) and, much as above, we get the commutative diagram

\begin{equation} \xymatrix{
\widehat{K} \ar@{^{(}->}[r]\ar@{^{(}->}[d]^{\cong}  & \widehat{G} \ar@{^{(}->}[d]^{\zeta} \ar@{->>}[r]^{f}& \widehat{\G}  \ar@{^{(}->}[d]^{\eta} \\
\widehat{K} \ar@{^{(}->}[r]& H \ar@{->>}[r]^{p}  & \phi(\widehat{\G}) } \end{equation} where $H = \zeta({\widehat{G}}) \leq \Z^m\rtimes SL_m(\Z)$. The horizontal sequence is a split extension, so by Lemma \ref{lem:ehsd} its excessive homology is given by the coinvariant homology of the fiber. We can now proceed as in the previous proof, using Relative Property (T) for the bottom row. 

To complete the proof of the statement, observe that the vanishing of virtual excessive homology for $G$ entails that it does not virtually algebraically fiber, as the BNS invariant of all its covers coincide with that of a free group or a surface, hence it is empty. But this implies by \cite[Theorem 5.3]{K20} that $G$ is not virtually RFRS as for such groups the only obstruction to virtual algebraic fibration is the vanishing of the first $\ell^2$--Betti number, which we have since there is an infinite index normal finitely generated non-trivial subgroup in all of our cases of $G$. 
\end{proof} 

Note, by contrast, that the extensions discussed in Theorem \ref{thm:freebyfreenv} and the potential counterexamples to Putman--Wieland conjecture are virtually residually $p$, in particular virtually residually finite solvable. 

Recall that given a prime $p$, a group is called {\em residually $p$} if the intersection of its $p$- power index normal subgroups is trivial. Residually $p$ groups are residually finite solvable (or RFS).

\begin{lemma}\label{lm:RFS}
	Let $G$ be a group that fits into a short exact sequence $1\to K\to G \stackrel{f}{\to} \G \to 1$. Suppose $K$ has trivial center and $K$, $\G$ are virtually residually $p$, then $G$ is virtually residually $p$.
\end{lemma}
\begin{proof}
	Without loss of generality, we can assume that $\G$ is already residually $p$. By \cite{Lub80}, we have that if $K$ is virtually residually $p$, then so is $\Aut(K)$. 
	Let $\zeta \colon G\to \Aut(K)$ be the representation of $G$ covering the monodromy map $\eta \colon \G \to \Out(K)$ as in \Cref{eq:bir}. 
	Since $K$ has trivial center, $\zeta$ is injective when restricted to $K$.
	Let $A \leq  \Aut(K)$ be a finite index subgroup of $\Aut(K)$ which is residually $p$.

	Let $\GG = \zeta^{-1}(A)\cap G \leq  G$.
	There is a short exact sequence $1\to \KK\to \GG\to \Ga\to 1$, where $\KK$ is a finite index subgroup of $K$. 
	
	Let $g\in \GG \smallsetminus\{e\}$. If $f(g)$ is non-trivial then we can find a $p$-group quotient of $\Ga$ where the image of $f(g)$ is non-trivial. Thus we have a $p$-group quotient of $\GG$ where the image of $g$ is non-trivial. 
	
	If $f(g) = e$, then $g\in \KK$. 
	In this case $\zeta(g)$ is non-trivial. 
	Thus we can find a $p$-group quotient of $A$, hence of $\GG$, under which the image of $\zeta(g)$ is non-trivial. 
	Thus in either case, we can find a $p$-group quotient where the image of $g$ is non-trivial and hence $\GG$ is residually $p$ and $G$ is virtually residually $p$.  
\end{proof}

\begin{remarks} \label{rem:rmks}
\begin{enumerate}
\item The core of the proof of Theorem \ref{thm:freebyfreenv} shows in more generality that if \label{cor:relT}
$N$ is a finitely generated group such that $\Aut(N)$ is finitely generated, and $(N\rtimes \Aut(N), N)$ has Relative Property (T), there exists a group of the form  $N\rtimes \G$ which does not algebraically fiber, where $\G$ can a be non-abelian free group or a surface group of sufficient rank. 
\item  With the result \cite[Theorem 6.1]{KW19} that $F_2 \rtimes F_n$ groups virtually algebraically fiber, we are left with the case of whether or not all groups of the form $F_3 \rtimes F_n$ virtually algebraically fiber. It is known that $\Aut(F_3)$ is large and hence does not have property (T). However it may be the case that $(F_3\rtimes \Aut(F_3), F_3)$ has Relative Property (T) and then, as remarked above, the same argument as in Theorem \ref{thm:freebyfreenv} could be applied. 
\end{enumerate}
\end{remarks}

Let $G$ be a group as in Equation (\ref{eq:ses}). In light of Kielak's result on virtual fibering \cite{K20} and the fact that $G$ has an infinite normal infinite-index subgroup (hence its first $L^2$ Betti number vanishes), $G$ would virtually algebraically fiber if $G$ is virtually RFRS. Although we know that many of these groups virtually algebraically fiber, we do not know if all of these are RFRS. More broadly, excluding the cases where the monodromy is finite (hence the extension is virtually a product), or the extension does not admit virtually excessive homology, we don't know if these virtually RFS (residually finite solvable) groups are virtually RFRS  (residually finite {\it rationally} solvable) or not. This is exemplified by the following  question:
\begin{question} Which free-by-free groups with nonabelian base and fiber and infinite monodromy are virtually RFRS? \end{question} 

In the hyperbolic case, cubulation would imply this.  Some surface-by-free groups were cubulated in \cite{MMS}, but other cases are widely unknown.  Therefore the following is of interest: 

\begin{question} Which (if any)  hyperbolic free-by-free groups are cubulated? \end{question} 

There are some homological consequences for extensions, such as those identified in Theorem \ref{thm:freebyfreenv}, which do not have virtually excessive homology (or equivalently are not virtually algebraically fibered).

\begin{proposition} Let $G$ be a surface-by-surface, a surface-by-free, or a free-by-free group without virtually excessive homology. Then for any $n \in \N$ and any prime $p$, there is a subgroup $\GG \leq  G$  of finite index at least $n$ such that $H_1(\GG)$ has nontrivial $p$-torsion. \end{proposition} 
\begin{proof} 
	By \Cref{lm:RFS}, the group $G$ is virtually residually $p$. 
	Because of the form of the statement, it is not restrictive to assume that $G$ itself is residually $p$. 
	Since $G$ is residually $p$, there exists a filtration $\{G_{i}\mid i \geq 0\}$ of finite index normal subgroups whose index is a power of $p$ with $\bigcap_{i} G_{i} = \{1\}$ and where the successive quotient maps $\a_{i}\colon  G_{i} \to G_i/G_{i+1} = S_i$  factorize through the maximal abelian quotient: 
	\begin{equation} \label{eq:rfs} \xymatrix@=4pt{ 
	& & G_{i+1} \ar@{^{(}->}[rr] & & G_{i} \ar[dr]  \ar@{->>}[rr]^{\a_{i}} & & S_i  & & 
		\\  & & & & & H_1(G_{i}) \ar[ur] &  &} \end{equation}
	Now let $\kappa \in K$ be an element with the property that $\kappa \in G_i \setminus G_{i+1}$;  denoting $K_i = K \cap G_i$ and combining the diagram in \Cref{eq:diag} with the residually $p$ assumption, we get the diagram
	
	\[ \xymatrix@=9pt{ 
		& & K_i \ar@{^{(}->}[rr]\ar[dd]_{\a_i} \ar[dr] & & G_i \ar'[d][dd]\ar[dr]  \ar@{->>}[rr]^{f_i} & & \G_i \ar[dr]\ar'[d][dd] & & 
		\\ & & & H_1(K_i)_{\G_i} \ar[dl] \ar[rr]_{\hspace{-2mm} \a_i} & & H_1(G_i) \ar@{->>}[rr]\ar[dl] &  & H_1(\G_i) \ar[dl] & &
		\\ &  & \a_{i}(K_{i}) \ar@{^{(}->}[rr] & & S_{i}  \ar@{->>}[rr] & & S_i/\a_{i}(K_{i}) & & } \]
	As $b_1(G_i) = b_1(\G_i)$, the image of $H_1(K_i)_{\G_i}$ in $H_1(G_i)$ is a torsion subgroup. And as $\a_i(\kappa) \neq 1 \in S_i$, the class $[\kappa] \in H_1(G_i)$ is nonzero, hence the torsion subgroup is nontrivial. 
	Moreover, $\a(\kappa)$ is non-trivial in $S_i$ and has order $p^l$ for some $l$. 
	We can consider the class $[\kappa]$ in $H_1(G_i)$, this has finite order and maps onto $\a(\kappa)$, thus we see that $[\kappa]$ has order $p^lr$ for some $r$. 
	We conclude that $H_1(G_i)$ contains an element of order $p$. 
\end{proof}

%=========================================================

\end{document}